\documentclass[a4paper,12pt]{amsart}
\usepackage{latexsym}\usepackage{amsmath}\usepackage{xspace}\usepackage{enumerate}\usepackage{amsfonts}\usepackage{amscd}\usepackage{syntonly}\usepackage{amssymb}\usepackage{caption}
\usepackage[T1]{fontenc}

\newtheorem{thm}{Theorem}

\newtheorem{lemma}[thm]{Lemma}
\newtheorem{prop}[thm]{Proposition}

\newtheorem{cor}[thm]{Corollary}
\newtheorem{claim}{Claim}

\newtheorem*{Rmk}{Remark}
\newtheorem*{Rmks}{Remarks}
\newtheorem{rmk}[thm]{Remark}

\newtheorem*{q}{Question}
\newenvironment{pf}[1][Proof]{\noindent \textbf{#1:} }{\hspace{\stretch{1}}}\newenvironment{enui}{\begin{enumerate}[(i)]}{\end{enumerate}}
\newenvironment{enua}{\begin{enumerate}[(a)]}{\end{enumerate}}

\newcommand{\BAR}[1]{{\overline{#1}}}

\newcommand\ga{\gamma}

\newcommand\eps{\varepsilon}
\newcommand\Om{\Omega}
\newcommand\om{\omega}

\newcommand\Si{\Sigma}
\newcommand\si{\sigma}

\newcommand\ze{\zeta}
\renewcommand\phi{\varphi}

\newcommand{\N}{\mathbb{N}}
\newcommand{\Z}{\mathbb{Z}}

\newcommand{\R}{\mathbb{R}}
\newcommand{\T}{\mathbb{T}}

\newcommand{\wt}[1]{\widetilde{#1}}

\newcommand\supp{\operatorname{supp}}

\newcommand\PP{\mathcal{P}}
\newcommand\hhat[1]{\widehat{#1}}
\newcommand\nn{\nonumber}

\newcommand\C{\mathbb C}

\newcommand\x{\times}
\newcommand\wo{\setminus}
\newcommand\sub{\subseteq}
\newcommand\one{\mathbf{1}}
\newcommand\Ham{\operatorname{Ham}}
\newcommand\Hamc{\operatorname{Ham_c}}

\newcommand\diam{\operatorname{diam}}
\newcommand\Diam{\operatorname{Diam}}
\newcommand\D{\mathbb{D}}
\newcommand\CP{\mathbb{C}\!\operatorname{P}}
\newcommand\RP{\mathbb{R}\!\operatorname{P}}

\newcommand\id{\operatorname{id}}
\newcommand\pr{\operatorname{pr}}

\newcommand\then{{\Longrightarrow}}

\newcommand\HH{\mathcal{H}}

\newcommand\Lip{\operatorname{Lip}}
\newcommand\A{\mathcal{A}}
\newcommand\WLOG{W.l.o.g.~}
\newcommand\Comm[1]{}
\newcommand\reff[1]{(\ref{#1})}
\newcommand{\VERT}{\vert\kern-0.25ex\vert\kern-0.25ex\vert}
\newcommand\wrt{w.r.t.~}
\newcommand\omst{\om_{\operatorname{st}}}
\newcommand\QHev{\operatorname{QH_{ev}}}
\newcommand\Cal{\operatorname{Cal}}
\newcommand\CAL{\hhat{\operatorname{Cal}}}

\title[Relative Hofer estimate]{A relative Hofer estimate and the asymptotic Hofer-Lipschitz constant}

\author{Michael Khanevsky and Fabian Ziltener}

\begin{document}


\maketitle

\begin{abstract}Let $(M,\om)$ be a symplectic manifold and $U\sub M$ an open subset. We study the natural inclusion of the compactly supported Hamiltonian group of $U$ in the compactly supported Hamiltonian group of $M$. The main result is an upper bound for this map in terms of the Hofer norms for $U$ and $M$.

Applications are upper bounds on the asymptotic Hofer-Lipschitz constant and the relative Hofer diameter of $U$. The first bound is often sharp and the second one is often sharp up to a factor of 2.
\end{abstract}

\tableofcontents

\section{Results}\label{sec:mot main}
\subsection{The main result}
The main result of this article is concerned with the following question. For simplicity, in this article \emph{manifold} means \emph{manifold without boundary}. Let $(M,\om)$ be a symplectic manifold. We denote by $C^\infty_c\big([0,1]\x M\big)$ the set of compactly supported real-valued functions on $[0,1]\x M$. For every $H\in C^\infty_c\big([0,1]\x M\big)$ we denote $H_t:=H(t,\cdot)$ and by $\phi_H=(\phi_H^t)_{t\in[0,1]}$ the Hamiltonian flow of $H$ \wrt $\om$. We define the \emph{compactly supported Hamiltonian group of $(M,\om)$} and the \emph{Hofer norms} on the sets of functions and the Hamiltonian group by
\begin{eqnarray}
\nn&\Hamc(M):=\Hamc(M,\om):=\big\{\phi_H^1\,\big|\,H\in C^\infty_c\big([0,1]\x M\big)\big\},&\\
\nn&\VERT\cdot\VERT:=\VERT\cdot\VERT^M_c:C^\infty_c\big([0,1]\x M\big)\to\R,&\\
\label{eq:VERT H VERT}&\VERT H\VERT:=\displaystyle\int_0^1\left(\max_MH_t-\min_MH_t\right)dt,&\\
\nn&\Vert\cdot\Vert^M_c:=\Vert\cdot\Vert^{M,\om}_c:\Hamc(M)\to\R,&\\
\nn&\Vert\phi\Vert^M_c:=\inf\big\{\VERT H\VERT\,\big|\,H\in C^\infty_c\big([0,1]\x M\big):\,\phi_H^1=\phi\big\}.&
\end{eqnarray}
Let $U\sub M$ be an open subset. Consider the natural inclusion
\begin{equation}\label{eq:phi wt phi}\Hamc(U)\ni\phi\mapsto\wt\phi\in\Hamc(M),\,\,\,\wt\phi(x):=\left\{\begin{array}{ll}
\phi(x),&\textrm{if }x\in U,\\
x,&\textrm{otherwise,}
\end{array}\right.
\end{equation}

\begin{q}How much does this map fail to be an isometry with respect to the Hofer norms for $U$ and $M$?
\end{q}

The main result of this article is the following theorem, which implies that the answer to this question is ``a lot'', if $U$ is small compared to $M$ in a suitable sense. To state this result, for $a>0$, we denote by $B^2(a),\BAR B^2(a)\sub\R^2$ the open and closed balls of radius $\sqrt{a/\pi}$, around 0. We denote by $\omst$ the standard symplectic form on $\R^{2n}$.
\begin{thm}[relative Hofer estimate]\label{thm:Vert wt phi} For every $\phi\in\Hamc(U)$ we have
\begin{equation}\label{eq:Vert wt phi}\Vert\wt\phi\Vert^M_c\leq\inf\left(2a+\frac{\Vert\phi\Vert^U_c}N\right),
\end{equation}
where $a\in(0,\infty)$ and $N\in\N:=\{1,2,\ldots\}$ run over all numbers for which there exists a symplectic manifold $(M',\om')$ and a symplectic embedding 
\[\psi:B^2(Na)\x M'\to M\]
(with respect to $\omst\oplus\om'$ and $\om$), satisfying
\[U\sub\psi(B^2(a)\x M').\]
\end{thm}
The estimate \eqref{eq:Vert wt phi} with the additive constant $2a$ replaced by $2Na$ and a factor of 2 in front of the second term is well-known.\footnote{This follows e.g.~the argument in the proof of \cite[Theorem C, p.~19]{PS}. The factor 2 in front of the second term comes from the fact that a generating Hamilton function may be both positive and negative. In the situation of that article this factor disappears, since the generating Hamilton function is positive.} 
The interest of Theorem \ref{thm:Vert wt phi} lies in the facts that the additive constant $2a$ does not depend on $N$ and there is no extra factor of 2.

The estimate \eqref{eq:Vert wt phi} is often asymptotically sharp as the Hofer norm of $\phi$ on $U$ tends to infinity. See Corollaries \ref{cor:Lipschitz} and \ref{cor:a N M' om'} below. On top of this, the additive constant is often sharp up to a factor of 2. See Corollary \ref{cor:Vert phi} and Proposition \ref{prop:a}. This improves the result of J.-C.~Sikorav for $\R^{2n}$ \cite[Proposition, p.~62]{Si} by a factor of 8. 

The strategy of proof of Theorem \ref{thm:Vert wt phi} is to adapt a version of Sikorav's method that was used by M.~Brandenbursky and K\k{e}dra in \cite{BK} to estimate the autonomous norm. 
This version of the method uses an algebraic identity of D.~Burago, S.~Ivanov, and L.~Polterovich \cite{BIP}.

To show our estimate, for a given Hamiltonian $H$ that generates $\phi$ we choose a compact subset $K$ of $M$, such that $[0,1]\x K$ contains the support of $H$. The trick is to choose a Hamiltonian diffeomorphism $\psi$ with Hofer norm less than $a$, such that the sets $\psi^i(K)$, $i=0,\ldots,N-1$, are disjoint. We now cut $\phi$ into time-pieces, which we transport to the regions $\psi^i(K)$. The resulting Hamiltonian diffeomorphism differs from $\wt\phi$ by some commutator with $\psi$. The estimate \eqref{eq:Vert wt phi} follows from this and the fact that $H$ can be chosen in such a way that
\begin{equation}\label{eq:c- }c_-\leq H\leq c_-+c,\end{equation}
where $c_-,c$ are constants, with $c$ arbitrarily close to the Hofer norm of $\phi$.

To show this fact, we choose a Hamiltonian for $\phi$ whose Hofer norm is close to that of $\phi$. We reparametrize the Hamiltonian in such a way that at each time its oscillation is less than $c$. The idea is now to shift the Hamiltonian by the product of a suitable function of time and some cut-off function on $M$, in such a way that the resulting Hamiltonian satisfies \eqref{eq:c- }.
\subsection{Application to the asymptotic Hofer-Lipschitz constant}\label{subsubsec:Hofer Lipschitz}
Theorem \ref{thm:Vert wt phi} has the following direct application. Let $(M,\om)$ be a symplectic manifold and $U$ an open subset of $M$. We define the \emph{asymptotic Hofer-Lipschitz constant of $(M,U,\om)$} to be
\begin{equation}\label{eq:Lip infty}\Lip^\infty(M,U):=\Lip^\infty(M,U,\om):=\end{equation}
\[\lim_{C\to\infty}\sup\left\{\frac{\Vert\wt\phi\Vert^M_c}{\Vert\phi\Vert^U_c}\,\Big|\,\phi\in\Hamc(U):\,\Vert\phi\Vert^U_c>C\right\}.\]
(Here our convention is that $\sup\emptyset:=0$.) This number can be understood as the asymptotic (for large distances) Lipschitz constant of the inclusion \eqref{eq:phi wt phi}, with respect to the Hofer distances for $U$ and $M$. It is the simplest interesting quantity comparing the two Hofer geometries, if $M$ is closed. (Compare to Subsection \ref{subsec:Lipschitz} in the appendix.)
\begin{cor}[upper bound on the asymptotic Hofer-Lipschitz constant]\label{cor:Lipschitz} Assume that there exists $a>0$, $N\in\N\cup\{\infty\}$, and a symplectic manifold $(M',\om')$, such that, defining $c:=Na$, we have 
\begin{equation}\label{eq:M cont B}M=B^2(c)\x M',\quad\om=\omst\oplus\om',\quad U=B^2(a)\x M'.\footnote{Here for $c=\infty$ we define $B^2(c):=\R^2$.}\end{equation}
Then we have 
\begin{equation}\label{eq:Lip infty N}\Lip^\infty(M,U)\leq\frac1N=\frac ac.
\end{equation}
\end{cor}
\begin{proof}This follows immediately from Theorem \ref{thm:Vert wt phi}.
\end{proof}
In particular, we have $\Lip^\infty(M,U)=0$, if $N=\infty$. Note that the obvious extension of the estimate \eqref{eq:Lip infty N} to a general triple $(M,\om,U)$ is false, hence the hypothesis that $M,\om$ and $U$ are products, cannot be dropped. (See Proposition \ref{prop:big asympt Lip} in Subsection \ref{subsec:general} in the appendix.)

The next result provides a sufficient criterion under which the estimate \eqref{eq:Lip infty N} is sharp. We call $(M,\om)$ \emph{(symplectically) aspherical} iff
\begin{equation}\label{eq:int S 2}\int_{S^2}u^*\om=0,\quad\forall u\in C^\infty(S^2,M).\end{equation}
We denote $2n:=\dim M$.
\begin{prop}[lower bound on the asymptotic Hofer-Lipschitz constant]\label{prop:Lip infty} The inequality
\begin{equation}\label{eq:Lip infty M U}\Lip^\infty(M,U)\geq\frac{\int_U\om^n}{\int_M\om^n}\end{equation}
holds if one of the following conditions is satisfied:
\begin{enua}
\item\label{prop:Lip infty:ex} The form $\om$ is exact and the symplectic volume of $M$ is finite.
\item\label{prop:Lip infty:asph} The manifold $M$ is closed, $\om$ is aspherical, and $U$ is displaceable in a Hamiltonian way.
\end{enua}
\end{prop}
In the case \reff{prop:Lip infty:ex} 
the proof of this result is based on the fact that in this situation the Calabi invariant descends to the Hamiltonian group. In the case \reff{prop:Lip infty:asph} the proof of this result is based on an argument by Y.~Ostrover used in the proof of \cite[Theorem 1.1]{Os}. Its key ingredient is a result of M.~Schwarz about action selectors (spectral invariants). We will deduce the following corollary from Proposition \ref{prop:Lip infty}.
\begin{cor}[lower bound on the asymptotic Hofer-Lipschitz constant]\label{cor:a N M' om'} Assume that there exist numbers $a>0$ and $c\geq2a$, and a closed aspherical symplectic manifold $(M',\om')$, such that \eqref{eq:M cont B} holds. Then we have
\begin{equation}\label{eq:Lip infty 1 N}\Lip^\infty(M,U)\geq\frac ac.
\end{equation}
\end{cor}
It follows that under the hypotheses of this corollary, the inequality \eqref{eq:Lip infty N} is sharp.
\begin{Rmk}\textnormal{In the case \reff{prop:Lip infty:asph} the proof of Proposition \ref{prop:Lip infty} given below can be extended to the more general settings of \cite[Theorems 1.1 and 1.3]{McDMono}, which provide conditions under which the (asymptotic) spectral invariants descend to $\Ham(M)$.
}
\end{Rmk}
\subsection{Application to the relative Hofer diameter}
Another immediate consequence of Theorem \ref{thm:Vert wt phi} is the following. Let $(M,\om)$ be a symplectic manifold and $U$ an open subset of $M$. We define the \emph{(extension) relative Hofer diameter of $U$ in $M$} to be
\begin{equation}\label{eq:Diam c}\Diam(U,M):=\Diam(U,M,\om):=\sup\big\{\Vert\wt\phi\Vert^M_c\,\big|\,\phi\in\Hamc(U)\big\}.
\end{equation}
This is the diameter of the distance function induced by the composition of the canonical extension homomorphism $\Hamc(U)\to\Hamc(M)$ with the Hofer norm, see Subsection \ref{subsec:rel Hofer diam} in the appendix.
\begin{cor}[upper bound on the relative Hofer diameter]\label{cor:Vert phi} Assume that there exists a symplectic manifold $(M',\om')$ and a number $a>0$, such that
\[(M,U,\om)=\big(\R^2\x M',B^2(a)\x M',\omst\oplus\om'\big).\]
Then we have
\[\Diam(U,M)\leq2a.\]
\end{cor}
\begin{proof}This follows immediately from Theorem \ref{thm:Vert wt phi}.
\end{proof}
In the case in which $M'=\R^{2n-2}$ for some $n\in\N$ J.-C.~Sikorav proved this estimate with the right hand side replaced by $16a$, see \cite[Proposition, p.~62]{Si}.\footnote{\cite[Proposition, p.~62]{Si} states that for every bounded subset $B$ of $\R^{2n}$ and every Hamiltonian isotopy $\phi$ with support in $B$ we have $\Vert\phi^1\Vert^M_c\leq8\Vert\psi\Vert^M_c$, where $\psi$ ranges over all compactly supported Hamiltonian diffeomorphisms of $\R^{2n}$, such that $B$ and $\psi(B)$ are separated by some hyperplane. However, the proof only shows the estimate with a factor $16$ instead of $8$. See also \cite[Theorem 10, Section 5.6]{HZ}, where the mistake is corrected.} Hence Corollary \ref{cor:Vert phi} improves Sikorav's result for $(M,U,\om)=\big(\R^{2n},B^2(a)\x\R^{2n-2},\omst\big)$ by a factor of 8.
\begin{Rmk}\textnormal{The \emph{absolute} Hofer diameter $\Diam(M,M)$ has been calculated for many symplectic manifolds. In all known examples it is infinite. For an overview and references, see \cite{McD_Ham}.
}
\end{Rmk}

The next result provides sufficient conditions under which Corollary \ref{cor:Vert phi} is sharp up to a factor of $2$. Let $(M,\om)$ be a symplectic manifold. We call a symplectic manifold $(M,\om)$ \emph{(geometrically) bounded}\label{bounded} iff there exist an almost complex structure $J$ on $M$ and a complete Riemannian metric $g$ such that the following conditions hold:
\begin{itemize}
\item The sectional curvature of $g$ is bounded and $\inf_{x\in M}\iota^g_x>0$, where $\iota^g_x$ denotes the injectivity radius of $g$ at the point $x\in M$. 
\item There exists a constant $C\in(0,\infty)$ such that 
\[|\om(v,w)|\leq C|v|\,|w|,\quad\om(v,Jv)\geq C^{-1}|v|^2,\]
for all $v,w\in T_xM$ and $x\in M$. Here $|v|:=\sqrt{g(v,v)}$.
\end{itemize}
\begin{prop}[lower bound on the relative Hofer diameter]\label{prop:a} Assume that there exist $(M',\om')$ and $a$ as in Corollary \ref{cor:Vert phi}. Suppose also that $(M',\om')$ is aspherical and geometrically bounded, and there exists a nonempty closed symplectic manifold $(X,\si)$, such that 
\[n:=\frac12\big(\dim M'-\dim X-2\big)\geq0,\quad B^2(2a)\x(B^2(a))^n\x X\sub M'.\]
Then we have 
\begin{equation}\label{eq:Diam c U}\Diam(U,M)\geq a.\end{equation}
\end{prop}
The proof of this result is based on a leafwise fixed point theorem for coisotropic submanifolds proved by the second author in \cite{ZiLeafwise}. 
\subsection{Organization of the article and notation}
In Section \ref{sec:proof:thm:Vert wt phi} we prove our main result, Theorem \ref{thm:Vert wt phi}. In Section \ref{proof:prop:Lip infty,cor:a N M' om'} we prove the lower bounds on the asymptotic Hofer-Lipschitz constant stated in Proposition \ref{prop:Lip infty} and Corollary \ref{cor:a N M' om'}. In Section \ref{proof:prop:a} we prove the lower bound on the relative Hofer diameter stated in Proposition \ref{prop:a}. The appendix contains some remarks, as well as a proof of Proposition \ref{prop:H H'}, which is used in the proof of Proposition \ref{prop:Lip infty}.

In the rest of this article we will use the abbreviated notation
\[\Vert\cdot\Vert:=\Vert\cdot\Vert^M_c:\Hamc(M)\to\R.\]
\subsection{Acknowledgements}
We thank Felix Schlenk for making us aware that the Hofer-Lipschitz constant defined in formula \eqref{eq:Lip M U om} below satisfies \mbox{$\Lip(M,U)\geq1$}. We are grateful to Leonid Polterovich for sharing Proposition \ref{prop:big asympt Lip} (under the assumption \reff{prop:big asympt Lip:CP}) with us. Finally, we would like to thank Dusa McDuff for valuable feedback and Peter Spaeth for useful discussions.
\section{Proof of Theorem \ref{thm:Vert wt phi} (relative Hofer estimate)}\label{sec:proof:thm:Vert wt phi}
For the proof of Theorem \ref{thm:Vert wt phi} we need the following. Let $(M,\om)$ be a symplectic manifold, $U$ an open subset of $M$ with compact closure, and $\phi$ a Hamiltonian diffeomorphism on $M$ that is generated by a function with support in $[0,1]\x U$.
\begin{lemma}[pinching the generating Hamiltonian]\label{le:H} For every real number $c>\Vert\phi|U\Vert$ there exists a real number $c_-$ and a smooth function $H:[0,1]\x M\to\R$ that has compact support and Hamiltonianly generates $\phi$, such that
\begin{equation}\label{eq:c-}c_-\leq H\leq c_-+c.\end{equation}
\end{lemma}
In order to prove this lemma we choose a Hamiltonian for $\phi$ whose Hofer norm is close to that of $\phi$. We reparametrize the Hamiltonian in such a way that at each time its oscillation is less than $c$. The idea is now to shift the Hamiltonian by the product of a suitable function of time and some cut-off function on $M$.
\begin{proof}[Proof of Lemma \ref{le:H}] Since $U$ has compact closure, by the smooth version of Urysohn's lemma there exists a smooth function $\rho:M\to[0,1]$ that has compact support and equals 1 on $U$. By our hypothesis $c>\Vert\phi|U\Vert$ there exists a smooth function $\wt H:[0,1]\x M\to\R$ that satisfies 
\begin{equation}\label{eq:phi wt H 1 phi}
\phi_{\wt H}^1=\phi,
\end{equation}
and has support in $[0,1]\x U$ and Hofer norm less than $c$. The functions $t\mapsto\min\wt H_t,\max\wt H_t$ are continuous. Therefore, by suitably smoothly reparametrizing $\wt H$ in time, we may assume that
\[\max\wt H_t-\min\wt H_t<c,\quad\forall t\in[0,1].\]
We choose a smooth function $f:[0,1]\to\R$, such that
\begin{equation}\label{eq:max wt H t}\max\wt H_t-c<f(t)<\min\wt H_t,\quad\forall t\in[0,1].\end{equation}
We define
\[c_-:=\int_0^1f(t)dt,\, F:[0,1]\x M\to\R,\,F_t:=(-f(t)+c_-)\rho,\, H:=\wt H+F.\]
To see that the function $H$ satisfies \eqref{eq:c-}, recall that $\wt H$ vanishes outside of $U$, and $\rho$ equals 1 on $U$ and takes values in $[0,1]$. The inequality $c_-\leq H$ follows from these facts and the inequality $f(t)<\min\wt H_t$ in \eqref{eq:max wt H t}. The inequality $H\leq c_-+c$ follows from these facts and the inequality $\max\wt H_t-c<f(t)$ in \eqref{eq:max wt H t}. (This inequality implies that $\max\wt H_t-c<c_-$.) Hence $H$ satisfies \eqref{eq:c-}.

For every $t\in[0,1]$ the derivative $d\wt H_t$ has support inside of $U$, and the derivative $dF_t$ has support outside of $U$. It follows that for every $t\in[0,1]$, we have
\[\phi_H^t=\phi_F^t\circ\phi_{\wt H}^t.\]
Since $\int_0^1\big(-f(t)+c_-\big)dt=0$, we have $\phi_F^1=\id$. It follows that $\phi_H^1=\phi_{\wt H}^1=\phi$, where we used \eqref{eq:phi wt H 1 phi}. Hence $H$ has the desired properties. This proves Lemma \ref{le:H}.
\end{proof}
\begin{proof}[Proof of Theorem \ref{thm:Vert wt phi}]\setcounter{claim}{0} \WLOG we may assume that there exist $a\in(0,\infty)$, $N\in\N$, and a symplectic manifold $(M',\om')$, such that
\begin{align}\label{eq:M}M&=B^2(Na)\x M',\\
\nn\om&=\omst\oplus\om',\\
\label{eq:U}U&=B^2(a)\x M'.
\end{align}
Let $\phi\in\Hamc(U)$ and
\[c>\Vert\phi\Vert\]
be a real number. We choose a function $H'\in C^\infty_c\big([0,1]\x U\big)$, such that
\[\phi_{H'}^1=\phi,\quad\VERT H'\VERT<c.\]
We choose an open subset $V\sub M$ whose closure is compact and contained in $U$, such that $[0,1]\x V$ contains the support of $H'$. By Lemma \ref{le:H} with $M,U$ replaced by $U,V$, there exists a real number $c_-$, a compact subset $K$ of $U$, and a smooth function $H:[0,1]\x U\to\R$ that has support contained in $[0,1]\x K$, such that
\[\phi_H^1=\phi\]
and the inequalities \eqref{eq:c-} holds. We define $\wt H:[0,1]\x M\to\R$ to be equal to $H$ on $[0,1]\x U$ and 0 outside of this set.

An elementary argument using (\ref{eq:M},\ref{eq:U}), shows that there exists $\psi\in\Hamc(M)$, such that
\begin{eqnarray}\label{eq:Vert psi}&\Vert\psi\Vert<a,&\\
\label{eq:K psi K}&K,\psi(K),\ldots,\psi^{N-1}(K)\textrm{ are (pairwise) disjoint},&
\end{eqnarray}
where $\psi^i:=\psi\circ\cdots\circ\psi$ ($i$ factors). We abbreviate 
\[\phi_i:=\phi_{\wt H}^{\frac iN},\quad\phi_{i,j}:=\psi^j\phi_i\psi^{-j},\quad\forall i,j\in\{0,\ldots,N-1\}.\]
We define
\begin{equation}\label{eq:chi}\chi:=\phi_{N-1,0}\phi_{N-2,1}\cdots\phi_{1,N-2},\end{equation}
where for simplicity we leave out the composition signs. We define $F:[0,1]\x M\to\R$ by
\begin{equation}\label{eq:F t x}F(t,x):=\left\{\begin{array}{ll}
\displaystyle\frac{H_{\frac{t+N-i-1}N}\circ\psi^{-i}(x)}N,&\textrm{on }\psi^i(K),\,\forall i\in\{0,\ldots,N-1\},\\
0,&\textrm{otherwise.}
\end{array}\right.\end{equation}
We denote by $\wt\phi:M\to M$ the map given by $\phi$ on $U$ and the identity outside $U$.
\begin{claim}\label{claim:phi F 1 =}We have
\begin{equation}\label{eq:phi F 1 =}\wt\phi=\phi_F^1\chi\psi\chi^{-1}\psi^{-1}.\end{equation}
\end{claim}
\begin{proof}[Proof of Claim \ref{claim:phi F 1 =}] We have
\begin{equation}\label{eq:psi chi psi-1}\psi\chi^{-1}\psi^{-1}=\phi_{1,N-1}^{-1}\cdots\phi_{N-1,1}^{-1}.
\end{equation}
Since $\phi_i$ equals the identity outside $K$, it follows from \eqref{eq:K psi K} that $\phi_{i,j}$ and $\phi_{i',j'}$ commute, if $j\neq j'$. Combining this with (\ref{eq:psi chi psi-1},\ref{eq:chi}), it follows that
\[\chi\psi\chi^{-1}\psi^{-1}=\left\{\begin{array}{ll}
\phi_{N-1,0}&\textrm{ on }K,\\
\phi_{N-i-1,i}\phi_{N-i,i}^{-1}&\textrm{ on }\psi^i(K),\,\forall i\in\{1,\ldots,N-1\},\\
\id&\textrm{ otherwise.}
\end{array}\right.\]
Using \eqref{eq:F t x}, equality \eqref{eq:phi F 1 =} follows. This proves Claim \ref{claim:phi F 1 =}.
\end{proof}
Using Claim \ref{claim:phi F 1 =}, we have
\begin{equation}\label{eq:Vert wt phi Vert}\Vert\wt\phi\Vert\leq\Vert\phi_F^1\Vert+\Vert\chi\psi\chi^{-1}\Vert+\Vert\psi^{-1}\Vert.\end{equation}
By the inequalities \eqref{eq:c-} we have
\[\max F-\min F\leq\frac cN\]
and therefore
\[\Vert\phi_F^1\Vert\leq\frac cN.\]
Combining this with \eqref{eq:Vert wt phi Vert}, the equalities
\[\Vert\chi\psi\chi^{-1}\Vert=\Vert\psi\Vert,\quad\Vert\psi^{-1}\Vert=\Vert\psi\Vert,\]
and \eqref{eq:Vert psi}, it follows that
\[\Vert\wt\phi\Vert<\frac cN+2a.\]
Since this holds for all $c>\Vert\phi\Vert$, the desired inequality \eqref{eq:Vert wt phi} follows. This proves Theorem \ref{thm:Vert wt phi}.
\end{proof}
\begin{Rmk}\textnormal{The idea of writing $\wt\phi$ as in \eqref{eq:phi F 1 =} comes from \cite[proof of the theorem on p.~64]{BK}, in which a given Hamiltonian diffeomorphism of $\R^{2n}$ is written as a product of autonomous pieces. The identity \eqref{eq:phi F 1 =} corresponds to the algebraic identity provided by the proof of \cite[Lemma 2.4]{BIP}.
}
\end{Rmk}
\section{Proofs of Proposition \ref{prop:Lip infty} and Corollary \ref{cor:a N M' om'} (lower bound on the asymptotic Hofer-Lipschitz constant)}\label{proof:prop:Lip infty,cor:a N M' om'}
In this section we prove Proposition \ref{prop:Lip infty} and Corollary \ref{cor:a N M' om'}. To treat the case \reff{prop:Lip infty:ex} in Proposition \ref{prop:Lip infty}, we need the following. Let $(M,\om)$ be a symplectic manifold. We denote $2n:=\dim M$ and define the \emph{Calabi invariant for $(M,\om)$ (on functions)} to be the map
\begin{eqnarray}\nn&\CAL:=\CAL_{(M,\om)}:C^\infty_c\big([0,1]\x M\big)\to\R,&\\
\label{eq:CAL H}&\CAL(H):=\displaystyle\int_0^1\left(\int_MH_t\om^n\right)dt.&
\end{eqnarray}
Assume now that $\om$ is exact.\footnote{Together with our standing assumption that $M$ has no boundary, this implies that each connected component of $M$ is noncompact.}
\begin{lemma}[Calabi invariant]\label{le:Cal} Let $H,H'\in C^\infty_c\big([0,1]\x M\big)$ be functions that generate the same Hamiltonian time-1 flow\footnote{By the Hamiltonian time-1 flow of $H$ we mean $\phi_H^1$.}. Then we have $\CAL(H)=\CAL(H')$.
\end{lemma}
\begin{proof} This follows from the definition of the Calabi homomorphism on $\Hamc(M)$ as in \cite[(10.3.2), p.~407]{MSIntro}, and from \cite[Lemma 10.3.4, p.~408]{MSIntro}, which links this definition with the above definition of $\CAL$.
\end{proof}
We define the \emph{Calabi homomorphism for $(M,\om)$} to be the map
\begin{equation}\label{eq:Cal}\Cal:=\Cal_{(M,\om)}:\Hamc(M)\to\R,\quad\Cal(\phi):=\CAL(H),\end{equation}
where $H\in C^\infty_c\big([0,1]\x M\big)$ is an arbitrary function, whose Hamiltonian time-1 flow equals $\phi$. By Lemma \ref{le:Cal} this map is well-defined, i.e., it does not depend on the choice of $H$. In the proof of Proposition \ref{prop:Lip infty} in the case \reff{prop:Lip infty:ex} we will use the following remark.
\begin{rmk}\label{rmk:sup inf int} Let $M$ be a (smooth) manifold, $\Om$ a volume form on $M$, and $F:M\to\R$ a continuous function, such that $0\in F(M)$. Then the following inequality holds:
\[\left(\sup_MF-\inf_MF\right)\int_M\Om\geq\int_MF\,\Om.\]
\end{rmk}
\begin{proof}[Proof of Proposition \ref{prop:Lip infty} in the case \reff{prop:Lip infty:ex}] 
Let $\phi\in\Hamc(M)$. Let $c\in\big(\Vert\phi\Vert,\infty\big)$. We choose $H\in C^\infty_c\big([0,1]\x M\big)$, such that $\phi_H^1=\phi$ and $c\geq\VERT H\VERT$. For every measurable subset $X\sub M$ we write $|X|:=\int_X\om^n$. We have
\begin{align*}c\geq&\VERT H\VERT\\
=&\int_0^1\big(\max_MH_t-\min_MH_t\big)dt\\
\geq&\frac1{|M|}\int_0^1\left(\int_MH_t\om^n\right)dt\\
&\textrm{(using Remark \ref{rmk:sup inf int} and our hypothesis that $|M|$ is finite)}\\
=&\frac{\Cal(\phi)}{|M|}.
\end{align*}
Since $c>\Vert\phi\Vert$ is arbitrary, it follows that
\begin{equation}\label{eq:Vert phi Vert Cal}\Vert\phi\Vert\geq\frac{\Cal(\phi)}{|M|}.
\end{equation}
Let now $C\in[1,\infty)$. We choose a function $H\in C^\infty\big(U,[0,C]\big)$ with compact support, such that
\begin{equation}\label{eq:int U H}\int_UH\om^n\geq(C-1)|U|.\end{equation}
We denote $\phi:=\phi_H^1$ and by $\wt\phi:M\to M$ the map given by $\phi$ on $U$ and the identity outside $U$. We have
\begin{align}\nn\Vert\wt\phi\Vert&\geq\frac{\Cal(\wt\phi)}{|M|}\qquad\textrm{(by \eqref{eq:Vert phi Vert Cal})}\\
\nn&=\frac{\displaystyle\int_UH\om^n}{|M|}\qquad\textrm{(by (\ref{eq:Cal},\ref{eq:CAL H}))}\\
\label{eq:Vert wt phi Vert Cal}&\geq(C-1)\frac{|U|}{|M|}\qquad\textrm{(using \eqref{eq:int U H})}
\end{align}
Since $0\leq H\leq C$, we have
\[\Vert\phi\Vert\leq\VERT H\VERT\leq C.\]
Combining this with \eqref{eq:Vert wt phi Vert Cal}, it follows that
\[\frac{\Vert\wt\phi\Vert}{\Vert\phi\Vert}\geq\frac{C-1}{C}\frac{|U|}{|M|}.\]
Using that $C$ is arbitrarily big, the inequality $\Vert\phi\Vert\geq\Vert\wt\phi\Vert$, and again \eqref{eq:Vert wt phi Vert Cal}, it follows that
\[\Lip^\infty(M,U)\geq\frac{|U|}{|M|}.\]
This proves Proposition \ref{prop:Lip infty} in the case \reff{prop:Lip infty:ex}.
\end{proof}

To prove Proposition \ref{prop:Lip infty} in the case \reff{prop:Lip infty:asph}, we will now adapt the proof of \cite[Theorem 1.1]{Os}, which is based on a result of M.~Schwarz. 

Let $(M,\om)$ be an aspherical symplectic manifold (i.e., \eqref{eq:int S 2} holds) and $H\in C^\infty\big([0,1]\x M\big)$. We define the \emph{action spectrum of $H$} as follows. We denote by $\D\sub\C$ the closed unit disk, and define the set of \emph{contractible $H$-periodic points} to be
\[\PP^\circ(H):=\big\{x_0\in M\,\big|\,\exists u\in C^\infty(\D,M):\,\phi_H^t(x_0)=u(e^{2\pi it}),\,\forall t\in[0,1]\big\}.\]
We define the \emph{$H$-twisted symplectic action of $x_0\in\PP^\circ(H)$} to be
\begin{equation}\label{eq:A H x0}\A_H(x_0):=-\int_\D u^*\om-\int_0^1H\big(t,\phi_H^t(x_0)\big)dt,\end{equation}
where $u\in C^\infty(\D,M)$ is any map satisfying $\phi_H^t(x_0)=u(e^{2\pi it})$, for every $t\in[0,1]$. It follows from asphericity of $(M,\om)$ that this number does not depend on the choice of $u$ and hence is well-defined. We define the \emph{action spectrum of $H$} to be
\[\Si_H:=\A_H(\PP^\circ(H))\sub\R.\]
We will use the following result to make sense of Proposition \ref{prop:Vert phi H 1 Si H} below.
\begin{prop}[action spectrum]\label{prop:Si H} Assume that $M$ is closed. Then $\Si_H$ is compact. 
\end{prop}
\begin{proof}[Proof of Proposition \ref{prop:Si H}] This is part of the statement of \cite[Proposition 3.7]{Schwa}.
\end{proof}
The proof of Proposition \ref{prop:Lip infty} in the case \reff{prop:Lip infty:asph} is based on the following result, which is a consequence of an argument of M.~Schwarz.
\begin{prop}[lower bound on Hofer-norm]\label{prop:Vert phi H 1 Si H} Assume that $(M,\om)$ is closed, connected, and aspherical. Then for every $H\in C^\infty\big([0,1]\x M\big)$ we have
\begin{equation}\label{eq:Vert phi H 1}\Vert\phi_H^1\Vert\geq\min\Si_H+\frac{\int_0^1\left(\int_MH_t\om^n\right)dt}{\int_M\om^n}.
\end{equation}
\end{prop}
\begin{Rmk}\textnormal{It follows from Proposition \ref{prop:Si H} that the minimum $\min\Si_F$ exists. Hence the right hand side of \eqref{eq:Vert phi H 1} makes sense.
}
\end{Rmk}
We call $F\in C^\infty\big([0,1]\x M\big)$ \emph{mean normalized (\wrt $\om$)} iff
\begin{equation}\label{eq:int M F t}\int_MF_t\om^n=0,\quad\forall t\in[0,1].
\end{equation}
We denote
\[\HH:=\left\{F\in C^\infty\big([0,1]\x M\big)\,\bigg|\,
F\textrm{ is mean normalized}\right\}.\]
\begin{proof}[Proof of Proposition \ref{prop:Vert phi H 1 Si H}]\setcounter{claim}{0} It follows from \cite[Theorem 12.4.4]{MSJ} and \cite[Proposition 3.1(i)]{McDMono} that there exists a map
\[c:\HH\to\R,\]
such that for every $F\in\HH$, we have
\begin{align}\label{eq:c H Si H}c(F)&\in\Si_F,\\
\label{eq:c H int}\Vert\phi_F^1\Vert&\geq c(F).\end{align} 
Namely, in the notation of \cite[Theorem 12.4.4]{MSJ} the spectral invariant $c(F):=\rho(\wt\phi_F;1)$ satisfies \eqref{eq:c H Si H} by \cite[Theorem 12.4.4, (Spectrality)]{MSJ} and the inequality
\begin{equation}\label{eq:Vert H Vert}\VERT F\VERT\geq c(F),\quad\forall F\in\HH,
\end{equation}
by \cite[Theorem 12.4.4, (Continuity), (12.4.5)]{MSJ}. Here we used the definition \eqref{eq:VERT H VERT} and the assumptions that $M$ is closed and $\om$ is aspherical, and therefore strongly semi-positive and rational on $\pi_2(M)$ (conditions \cite[(8.5.1), (12.4.1)]{MSJ}). Using again our hypothesis that $\om$ is aspherical, it follows from \cite[Proposition 3.1(i)]{McDMono} that for all $F,F'\in\HH$ satisfying $\phi_F^1=\phi_{F'}^1$ we have $c(F)=c(F')$. (This means that the spectral invariant $c$ descends from the universal cover of $\Ham(M)$ to $\Ham(M)$.) Combining this with \eqref{eq:Vert H Vert}, inequality \eqref{eq:c H int} follows.\\

Let now $H\in C^\infty\big([0,1]\x M\big)$. We define 
\[f:[0,1]\to\R,\quad f(t):=\frac{\int_MH_t\om^n}{\int_M\om^n},\]
and $F:[0,1]\x M\to\R$ by $F_t(x):=F(t,x):=H_t(x)-f(t)$. By straight-forward arguments this function is mean normalized, generates $\phi_H^1$, and satisfies
\[\Si_F=\Si_H+\int_0^1f(t)dt.\]
Inequality \eqref{eq:Vert phi H 1} follows from this and (\ref{eq:c H Si H},\ref{eq:c H int}). This proves Proposition \ref{prop:Vert phi H 1 Si H}.
\end{proof}
In the proof of Proposition \ref{prop:Lip infty} in the case \reff{prop:Lip infty:asph} we will also use the following result, which is due to Y.~Ostrover. Let $(M,\om)$ be a symplectic manifold and $H,H'\in C^\infty_c([0,1]\x M)$. We denote by $H\#H':[0,1]\x M\to\R$ the time-concatenation of $H$ and $H'$, given by 
\[(H\#H')_t:=\left\{\begin{array}{ll}
2H^{2t},&\textrm{if }t\in[0,\frac12],\\
2{H'}^{2t-1},&\textrm{if }t\in(\frac12,1].
\end{array}\right.\]
\begin{prop}[action for concatenated Hamiltonian]\label{prop:H H'} Assume that $H_t,H'_t=0$ for $t$ in some neighbourhood of $\{0,1\}$, and defining $X:=\bigcup_{t\in[0,1]}\supp H_t$ \footnote{Here $\supp$ denotes the support of a function.}, we have
\begin{equation}\label{eq:phi H'}\phi_{H'}^1(X)\cap X=\emptyset.\end{equation}
Then the following holds:
\begin{enui}\item\label{prop:H H':PP}
\begin{equation}\label{eq:PP circ}\PP^\circ(H\#H')=\PP^\circ(H').\end{equation}
\item\label{prop:H H':A} If $(M,\om)$ is aspherical then we have
\begin{equation}\label{eq:A H H'}\A_{H\#H'}(x_0)=\A_{H'}(x_0),\quad\forall x_0\in\PP^\circ(H').
\end{equation}
\end{enui}
\end{prop}
This result follows from the proof of \cite[Proposition 2.2]{Os}. For the convenience of the reader we prove it in the appendix on page \pageref{proof:prop:H H'}.
\begin{proof}[Proof of Proposition \ref{prop:Lip infty} in the case \reff{prop:Lip infty:asph}]\setcounter{claim}{0} Without loss of generality, we may assume that $M$ is connected and $U\neq\emptyset$. For every measurable subset $X\sub M$ we write $|X|:=\int_X\om^n$. Let $C>0$ and 
\begin{equation}\label{eq:c c 0}c<c_0:=\frac{|U|}{|M|}\end{equation}
be a positive constant. We denote by $\wt\phi:M\to M$ the map given by $\phi$ on $U$ and the identity outside $U$.
\begin{claim}\label{claim:phi} There exists $\phi\in\Hamc(U)$ such that 
\begin{equation}\label{eq:C c}\Vert\wt\phi\Vert\geq\max\{C,c\Vert\phi\Vert\}.
\end{equation}
\end{claim}
\begin{proof}[Proof of Claim \ref{claim:phi}] By hypothesis there exists a function $F\in C^\infty\big([0,1]\x M\big)$ such that
\begin{equation}\label{eq:phi F 1}\phi_F^1(U)\cap U=\emptyset.
\end{equation}
Reparametrizing $F$, we may assume that $F_t=0$ for $t$ in some neighbourhood of $\{0,1\}$. Furthermore, replacing $F_t$ by $F_t-\int_MF_t\om^n/|M|$, we may assume that $F$ is mean normalized, i.e., it satisfies \eqref{eq:int M F t}. We choose a compact subset $K\sub U$ such that 
\begin{equation}\label{eq:frac int K om n}\frac{|K|}{|M|}>c.
\end{equation}
Furthermore, we choose a smooth function $H_0:U\to[0,1]$ with compact support, such that
\begin{equation}\label{eq:H 0 K}H_0|_K=1.\end{equation}
We define
\begin{equation}\label{eq:t 0}t_0:=\max\left\{\frac{\VERT F\VERT-\min\Si_F}{\displaystyle\frac{|K|}{|M|}-c},\frac Cc\right\}.
\end{equation}
It follows from \eqref{eq:frac int K om n} that $t_0<\infty$. We define 
\[\phi:=\phi_{t_0H_0}^1.\]
\begin{claim}\label{claim:eq:C c} This map satisfies inequality \eqref{eq:C c}.
\end{claim}
\begin{pf}[Proof of Claim \ref{claim:eq:C c}] We choose a function $f\in C^\infty\big([0,1],[0,1]\big)$ such that $f=i$ in a neighbourhood of $i$, for $i=0,1$. We define 
\begin{equation}\label{eq:H f'}H:[0,1]\x M\to\R,\quad H_t(x):=\left\{\begin{array}{ll}f'(t)t_0H_0(x),&\textrm{if }x\in U,\\
0,&\textrm{otherwise.}\end{array}\right.
\end{equation}
We have
\begin{equation}\label{eq:phi H 1 wt phi}\phi_H^1=\wt\phi.
\end{equation}
Using \eqref{eq:phi F 1}, the fact that the support of $H_0$ is contained in $U$, and asphericity of $(M,\om)$, the hypotheses of Proposition \ref{prop:H H'} with $H':=F$ are satisfied. Hence applying this proposition, it follows that 
\begin{equation}\label{eq:Si H F}\Si_{H\#F}=\Si_F.
\end{equation}
Applying Proposition \ref{prop:Vert phi H 1 Si H}, we have
\begin{equation}\label{eq:Vert phi H F}\Vert\phi_{H\#F}^1\Vert\geq\min\Si_{H\#F}+\frac{\displaystyle\int_0^1\left(\int_M(H\#F)_t\om^n\right)dt}{|M|}.\end{equation}
Using the triangle inequality and the fact $\Vert\phi_F^1\Vert\leq\VERT F\VERT$, we have
\begin{equation}\label{eq:Vert phi H 1 F}\Vert\phi_H^1\Vert\geq\Vert\phi_{H\#F}^1\Vert-\VERT F\VERT.
\end{equation}
We have 
\begin{align}\nn\int_0^1\left(\int_M(H\#F)_t\om^n\right)dt&=\int_0^1\left(\int_MH_t\om^n\right)dt+\int_0^1\left(\int_MF_t\om^n\right)dt\\
\label{eq:int 0 1 int M H F}&\geq t_0|K|+0\qquad\textrm{(using (\ref{eq:H f'},\ref{eq:H 0 K},\ref{eq:int M F t})),}\end{align}
\begin{align}\nn\Vert\wt\phi\Vert&=\Vert\phi_H^1\Vert\qquad\textrm{(using \eqref{eq:phi H 1 wt phi})}\\
\nn&\geq\min\Si_{H\#F}-\VERT F\VERT+t_0\frac{|K|}{|M|}\qquad\textrm{(using (\ref{eq:Vert phi H 1 F},\ref{eq:Vert phi H F},\ref{eq:int 0 1 int M H F}))}\\
\label{eq:Vert phi H c t 0}&\geq ct_0\qquad\textrm{(using (\ref{eq:Si H F},\ref{eq:t 0})).}
\end{align}
Using again \eqref{eq:t 0}, it follows that 
\begin{equation}\label{eq:Vert phi C}\Vert\wt\phi\Vert\geq C.
\end{equation}
Condition \eqref{eq:H 0 K}, the fact $K\neq\emptyset$, and the inequality $H_0\leq1$ imply that $\max_UH_0=1$. Since $H_0$ has compact support and satisfies $H_0\geq0$, we have $\min_UH_0=0$. Using \eqref{eq:H f'}, the fact $f(i)=i$, for $i=0,1$, and the Fundamental Theorem of Calculus, it follows that
\[\VERT H|[0,1]\x U\VERT=t_0.\]
Since $\phi=\phi_{H|[0,1]\x U}^1$, it follows that
\[\Vert\phi\Vert\leq t_0.\]
Combining this with \eqref{eq:Vert phi H c t 0} and \eqref{eq:Vert phi C}, inequality \eqref{eq:C c} follows. This proves Claim \ref{claim:eq:C c} and hence Claim \ref{claim:phi}. \end{pf}\end{proof}
We choose a map $\phi$ as in Claim \ref{claim:phi}. The fact $\Vert\phi\Vert\geq\Vert\wt\phi\Vert$ and inequality \eqref{eq:C c} imply that $\Vert\phi\Vert\geq C$. Inequality \eqref{eq:C c} also implies that $\Vert\wt\phi\Vert/\Vert\phi\Vert\geq c$. It follows that 
\[\Lip^\infty(M,U)\geq c.\]
Since $c<c_0$ (as defined in \eqref{eq:c c 0}) is arbitrary, the estimate \eqref{eq:Lip infty M U} follows. This completes the proof of Proposition \ref{prop:Lip infty} in the case \reff{prop:Lip infty:asph}.
\end{proof}
\begin{proof}[Proof of Corollary \ref{cor:a N M' om'}]\setcounter{claim}{0} We choose an area form $\si$ on the two-torus $\T^2$ such that $\int_{\T^2}\si=c$, and a symplectic embedding $\psi:B^2(c)\to\T^2$. Let $\eps>0$. We define 
\[(M,U,\om):=\big(\T^2\x M',\psi(B^2(a-\eps))\x M',\si\oplus\om'\big).\]
Then the hypotheses of Proposition \ref{prop:Lip infty} are satisfied. (That the subset $U\sub M$ is displaceable in a Hamiltonian way, follows from our hypothesis $c\geq2a$ and the fact that every open two-dimensional ball of area less than $a$ is displaceable inside every ball of area $2a$.) Therefore, applying this theorem, it follows that 
\[\Lip^\infty(M,U)\geq\frac{\int_U\om^n}{\int_M\om^n}=\frac{(a-\eps)\int_{M'}{\om'}^{n'}}{c\int_{M'}{\om'}^{n'}},\]
where $2n':=\dim M'$. Since $\eps>0$ is arbitrary, the claimed inequality \eqref{eq:Lip infty 1 N} follows. This proves Corollary \ref{cor:a N M' om'}.
\end{proof}
\section{Proof of Proposition \ref{prop:a} (lower bound on the relative Hofer diameter)}\label{proof:prop:a}
In the proof of Proposition \ref{prop:a} we will use the following definition. Let $(M,\om)$ be a symplectic manifold and $N\sub M$ a coisotropic submanifold. We define the \emph{action spectrum} and the \emph{minimal area} of $(M,\om,N)$ as
\begin{eqnarray}\nn&S(M,\om,N):=&\\
\nn&\left\{\displaystyle\int_\D u^*\om\,\bigg|\,u\in C^\infty(\D,M):\,\exists\textrm{ isotropic leaf }F\sub N:\,u(S^1)\sub F\right\},&
\end{eqnarray}
\[A(M,\om,N):=\inf\big(S(M,\om,N)\cap (0,\infty)\big)\in[0,\infty].\]
Furthermore, for $n\in\N$ and $a>0$ we denote by $S^{2n-1}(a)\sub\R^{2n}$ the sphere of radius $\sqrt{a/\pi}$, around 0.

\begin{proof}[Proof of Proposition \ref{prop:a}]\setcounter{claim}{0} Let $\eps>0$. We define 
\[N:=S^1(a-\eps)\x S^1(a-\eps)\x S^{2n-1}(a-\eps)\x X.\]
This is a closed and regular coisotropic submanifold of $U$. We choose a map $\phi_0\in\Hamc(B^2(2a))$ such that 
\begin{equation}\label{eq:phi 0}\phi_0(S^1(a-\eps))\cap S^1(a-\eps)=\emptyset.\end{equation}
(That there exists such a map follows the fact that every open two-dimensional ball of area less than $a$ is displaceable inside every ball of area $2a$.) Since $N$ is compact, by a cutoff argument there exists a map $\phi\in\Hamc(U)$ such that $\phi=\big(\id_{\R^2}\x\phi_0\x\id_{(B^2(a))^n\x X}\big)$ on $N$. (See for example \cite[Lemma 35]{SZSmall}.) It follows from \eqref{eq:phi 0} that 
\begin{equation}\label{eq:phi N N}\phi(N)\cap N=\emptyset.\end{equation} 
We define $V:=\R^2\x B^2(2a)\x(B^2(a))^n\x X$. Since by hypothesis, $(M',\om')$ is aspherical, the same holds for $(X,\si)$. Hence it follows from \cite[Remark 31, Lemma 30, and Proposition 34]{SZSmall} that 
\begin{equation}\label{eq:A V}A\big(V,\om|_V,N\big)=a-\eps.\end{equation}
Using again that $(M',\om')$ is aspherical and \cite[Lemma 33]{SZSmall}, we have 
\[A(M,\om,N)=A\big(V,\om|_V,N\big).\]
We denote by $\wt\phi:M\to M$ the map given by $\phi$ on $U$ and the identity outside $U$. Combining this with \eqref{eq:A V} and using \eqref{eq:phi N N} and geometric boundedness of $(M',\om')$, it follows from \cite[Theorem 1]{ZiLeafwise} that $\Vert\wt\phi\Vert\geq a-\eps$. Since $\eps>0$ is arbitrary, it follows that 
\[\Vert\wt\phi\Vert\geq a.\]
The inequality \eqref{eq:Diam c U} follows from this. This proves Proposition \ref{prop:a}.
\end{proof}
\appendix
\section{Remarks on the relative Hofer diameter, Hofer-Lipschitz constants, and Corollary \ref{cor:Lipschitz}}\label{sec:rmk}
\subsection{Hofer-Lipschitz constants}\label{subsec:Lipschitz} 
Let $(M,\om)$ be a symplectic manifold and $U\sub M$ an open subset. Instead of $\Lip^\infty(M,U)$ (as defined in \eqref{eq:Lip infty}), consider the \emph{Hofer-Lipschitz constant} of $(M,U,\om)$, which we define as
\begin{equation}\label{eq:Lip M U om}\Lip(M,U):=\Lip(M,U,\om):=\sup\left\{\frac{\Vert\wt\phi\Vert^M_c}{\Vert\phi\Vert^U_c}\,\bigg|\,\id\neq\phi\in\Hamc(U)\right\}.\end{equation}
This is the Lipschitz constant of the natural inclusion \eqref{eq:phi wt phi} \wrt the Hofer norms for $U$ and $M$. By \cite[Theorem 1.1]{LMGeo} every $\phi\in\Hamc(U)$ other than $\id$ has positive Hofer norm. Hence this definition makes sense. However, if $M$ is closed and $U\neq\emptyset$ then 
\begin{equation}\label{eq:Lip M U om = 1}\Lip(M,U)=1,\end{equation}
hence this number is uninteresting. To see that \eqref{eq:Lip M U om = 1} holds, note that without loss of generality, we may assume that $M$ is connected. By definition, we have $\Lip(M,U)\leq1$. Furthermore\footnote{We were made aware of the following argument by F.~Schlenk.}, let $H\in C^\infty_c(U)$ be a non-constant function. We define $\wt H:M\to\R$ by $\wt H(x):=H(x)$, if $x\in U$, and $\wt H(x):=0$, otherwise. It follows from \cite[Theorem 1.6(i)]{McD_geometric} that there exists $t_0>0$, such that 
\[\Vert\phi_{\wt H}^{t_0}\Vert^M_c=t_0\left(\max_M\wt H-\min_M\wt H\right)\geq\Vert\phi_H^{t_0}\Vert^U_c.\]
It follows that $\Lip(M,U)\geq1$, and therefore, equality \eqref{eq:Lip M U om = 1} holds. 
\subsection{Corollary \ref{cor:Lipschitz} (upper bound on the asymptotic Hofer-Lipschitz constant)}\label{subsec:general}
In view of the estimate \eqref{eq:Lip infty N}, it is natural to ask the following question.
\begin{q} Does there exist a constant $C>0$, such that for every symplectic manifold $(M,\om)$ of finite volume and every open subset $U\sub M$, the estimate
\begin{equation}\label{eq:Lip infty M U C}\Lip^\infty(M,U)\leq C\frac{\int_U\om^n}{\int_M\om^n}
\end{equation}
holds, where $2n:=\dim M$?
\end{q}
The answer is negative, even if we allow the constant $C$ to depend on the symplectic manifold. This follows from the next result, which in the case \reff{prop:big asympt Lip:surf} is based on a technique by F.~Lalonde and D.~McDuff \cite[proof of Lemma 5.7, p.~64]{LMHofer}, and in case \reff{prop:big asympt Lip:CP} is due to L.~Polterovich (private communication).

\begin{prop}[big asymptotic Lipschitz constant]\label{prop:big asympt Lip} The equality
\begin{equation}\label{eq:Lip infty M U 1}\Lip^\infty(M,U)=1\end{equation}
holds if $(M,\om,U)$ is given by one of the following:
\begin{enua}
\item\label{prop:big asympt Lip:surf} $(M,\om)$ is a two-dimensional symplectic manifold 
and $U\sub M$ an open neighbourhood of some non-contractible embedded circle in $M$.
\item\label{prop:big asympt Lip:CP} $M$ is the complex projective space $\CP^n$ for some $n\in\N$, $\om$ the Fubini-Studi form, and $U\sub M$ an open neighbourhood of the real projective space $\RP^n$ (embedded in $\CP^n$ in the standard way).
\end{enua}
\end{prop}
\begin{Rmks}\textnormal{\begin{itemize}\item Since we may choose $U$ to have arbitrary small volume in these examples, it follows that the bound \eqref{eq:Lip infty M U C} does not hold.
\item The equality \eqref{eq:Lip infty M U 1} means that there are arbitrarily Hofer-large Hamiltonian diffeomorphisms on $U$ whose Hofer norm does almost not shrink when trivially extending the diffeomorphism to $M$. This equality is optimal, since $\Lip^\infty(M,U)$ is always bounded above by 1.
\item The set $U$ in these examples is non-displaceable, since the same holds for the circle and $\RP^n$, respectively. Hence the statement of Proposition \ref{prop:Lip infty} continues to hold for some non-aspherical symplectic manifolds and some small non-displaceable subsets $U$. 
\end{itemize}
}
\end{Rmks}

In the proof of Proposition \ref{prop:big asympt Lip} in the case \reff{prop:big asympt Lip:surf} we will use the following.
\begin{lemma}[fundamental group of surface]\label{le:fund gp surf} If a connected real surface $M$ is not diffeomorphic to the real projective space $\RP^2$ then its fundamental group is torsionfree.
\end{lemma}
\begin{proof}[Proof of Lemma \ref{le:fund gp surf}] Assume first that $M$ is noncompact. Then its fundamental group is free, see e.g.~\cite[4.2.2 Infinite Surfaces, p.~142]{St}. Therefore it is torsionfree. 

If $M$ is diffeomorphic to $S^2$, then the statement of the lemma is true.

Assume now that $M$ is compact (without boundary) and not diffeomorphic to $S^2$ or $\RP^2$. Then its universal cover is diffeomorphic to $\R^2$. Here the nonorientable case can be reduced to the orientable case, by using that the Euler characteristic of the orientation double cover of $M$ equals twice the Euler characteristic of $M$.

It now follows from the proof of \cite[Lemma 4.1, p.~5]{Lu} that the fundamental group of $M$ is torsionfree. This proves Lemma \ref{le:fund gp surf}.
\end{proof}

\begin{proof}[Proof of Proposition \ref{prop:big asympt Lip} in the \textbf{case} \reff{prop:big asympt Lip:surf}]\setcounter{claim}{0} Let $(M,\om,U)$ be as in \reff{prop:big asympt Lip:surf}. In order to prove equality \eqref{eq:Lip infty M U 1}, it suffices to prove the inequality
\begin{equation}\label{eq:Lip infty M U geq}\Lip^\infty(M,U)\geq1.\end{equation}
We choose a noncontractible embedded circle $L$ in $M$ that is contained in $U$. We also choose a universal cover $\pi:\wt M\to M$. We equip $\wt M$ with the symplectic form $\wt\om:=\pi^*\om$. Let $C\in[4,\infty)$.
\begin{claim}\label{claim:area H}
\begin{enui}\item\label{claim:area H:area} The total area of $\wt\om$ is infinite.
\item\label{claim:area H:H} There exists a function $H\in C^\infty\big(M,[0,C]\big)$ with compact support contained in $U$, and compact submanifolds\footnote{with boundary} $\wt K^\pm$ of $\wt M$ that are symplectomorphic to $\BAR B^2(C-3)$, such that the following holds. If $\wt\phi:\wt M\to\wt M$ is a continuous lift of $\phi_H^1$ in the sense that $\pi\circ\wt\phi=\phi_H^1\circ\pi$, then $\wt\phi$ displaces $\wt K^+$ or $\wt K^-$.
\end{enui}
\end{claim}
\begin{proof}[Proof of Claim \ref{claim:area H}] We equip $\R\x\R$ and $(\R/\Z)\x\R$ with the canonical symplectic forms. By Weinstein's Lagrangian neighbourhood theorem there exist $a\in(0,\infty)$, an open neighbourhood $V$ of $L$ that is contained in $U$, and a symplectomorphism $\psi:(\R/\Z)\x(-a,a)\to V$, such that $\psi\big((\R/\Z)\x\{0\}\big)=L$. 

We denote by $\pi':\R\x(-a,a)\to(\R/\Z)\x(-a,a)$ the canonical projection. 
We choose a map $\wt\psi\in C^\infty\big(\R\x(-a,a),\wt M\big)$ 
satisfying
\begin{equation}\label{eq:pi wt psi}\pi\circ\wt\psi=\psi\circ\pi'.\end{equation}
By our hypothesis that $L$ is noncontractible, the condition $\psi\big((\R/\Z)\x\{0\}\big)=L$, and Lemma \ref{le:fund gp surf}, the map $\wt\psi$ is injective and therefore a symplectic embedding. It follows that the image of $\wt\psi$ has infinite area. Statement \reff{claim:area H:area} follows.\\

To prove \reff{claim:area H:H}, we choose a function $f\in C^\infty\big((-a,a),[0,C]\big)$ with compact support, such that
\begin{equation}\label{eq:f p}f(p)=\frac{C|p|}a\textrm{ on }\left(-a\frac{C-1}C,-\frac aC\right)\cup\left(\frac aC,a\frac{C-1}C\right).\end{equation}
We denote by $\pr:(\R/\Z)\x(-a,a)\to(-a,a)$ the canonical projection. We define the function $H:M\to\R$ by
\begin{equation}\label{eq:H}H:=\left\{\begin{array}{ll}
f\circ\pr\circ\psi^{-1}&\textrm{on }V,\\
0&\textrm{otherwise.}
\end{array}\right.
\end{equation}
We denote
\begin{equation}\label{eq:U+}U^+:=\left(0,\frac Ca\right)\x\left(\frac aC,a\frac{C-1}C\right),\quad U^-:=-U^+=\big\{-(q,p)\,\big|\,(q,p)\in U^+\big\},\end{equation}
and choose a compact submanifold
\begin{equation}\label{eq:K pm}K^\pm\sub U^\pm\end{equation} 
that is symplectomorphic to $\BAR B^2(C-3)$. We define
\begin{equation}\label{eq:wt K pm}\wt K^\pm:=\wt\psi(K^\pm).\end{equation}
Since $\wt\psi$ is a symplectic embedding, $\wt K^\pm$ is symplectomorphic to $\BAR B^2(C-3)$.

Let $\wt\phi:\wt M\to\wt M$ be a continuous lift of $\phi_H^1$. For every $c\in\R$ we define the shift map $s_c:\R\x(-a,a)\to\R\x(-a,a)$ by $s_c(q,p):=(q+c,p)$. We denote by $\pr':\R\x(-a,a)\to(-a,a)$. It follows from the equality $\pi\circ\wt\phi=\phi_H^1\circ\pi$ and (\ref{eq:pi wt psi},\ref{eq:H}) that there exists an integer $N$, such that
\[\wt\psi^{-1}\circ\wt\phi\circ\wt\psi=s_N\circ\phi_{f\circ\pr'}^1.\]
Using (\ref{eq:f p},
\ref{eq:U+},\ref{eq:K pm},\ref{eq:wt K pm}), it follows that $\wt\phi$ displaces $\wt K^+$ or $\wt K^-$. This proves \reff{claim:area H:H} and completes the proof of Claim \ref{claim:area H}.
\end{proof}
We choose $H$ and $\wt K^\pm$ as in part \reff{claim:area H:H} of this claim. Let $F\in C^\infty_c\big([0,1]\x M\big)$ be such that
\[\phi_F^1=\phi_H^1.\]
The time-1 flow of $F\circ\pi$ is well-defined on $\wt M$ and lifts the flow $\phi_F^1$. Therefore, by the conclusion of part \reff{claim:area H:H} of Claim \ref{claim:area H} the map $\phi_{F\circ\pi}^1$ displaces $\wt K^+$ or $\wt K^-$.

Since $M$ admits a noncontractible embedded circle (namely $L$), it is not diffeomorphic to $S^2$. Using that $M$ is orientable, it follows that its universal cover is diffeomorphic to $\R^2$. Using Claim \ref{claim:area H}\reff{claim:area H:area} and a result of R.~Greene and K.~Shiohama \cite[Theorem 1]{GS}\footnote{This result is based on Moser isotopy.}, it follows that $(\wt M,\wt\om)$ is symplectomorphic to $\R^2$ with the standard form. Therefore $(\wt M,\wt\om)$ is geometrically bounded.

Since $\wt K^\pm$ is symplectomorphic to $\BAR B^2(C-3)$ and $\phi_{F\circ\pi}^1$ displaces $\wt K^+$ or $\wt K^-$, the sharp energy-Gromov-width inequality\footnote{See e.g.~\cite[Corollary 3]{SZSmall}.} therefore implies that
\[\VERT F\circ\pi\VERT\geq C-3.\]
Since $\VERT F\VERT=\VERT F\circ\pi\VERT$, it follows that
\[\Vert\phi_H^1\Vert\geq C-3.\]
Since $0\leq H\leq C$, we have
\[\big\Vert\phi_{H|U}^1\big\Vert\leq\VERT H|U\VERT\leq C.\]
(The flow $\phi_{H|U}^1$ is well-defined on $U$, since $H$ has support contained in $U$.) Since $C\geq4$ is arbitrary, the inequality \eqref{eq:Lip infty M U geq} follows. This proves Proposition \ref{prop:big asympt Lip} in the case \reff{prop:big asympt Lip:surf}.
\end{proof}
\begin{Rmk}\textnormal{The above proof technique is based on the proof of \cite[Lemma 5.7, p.~64]{LMHofer}.
}
\end{Rmk}

In the proof of Proposition \ref{prop:big asympt Lip} in the case \reff{prop:big asympt Lip:CP} with $n\geq2$ we will use the following. For every topological space $X$, abelian group $A$, and integer $k$ we denote by $H_k(X;A)$ the $k$-th homology of $X$ with coefficients in $A$.
\begin{lemma}[map on homology induced by inclusion of real projective space]\label{le:homol incl real proj} For every $n\in\N$ and $k\in\left\{0,\ldots,\lfloor\frac n2\rfloor\right\}$ the map 
\begin{equation}\label{eq:H k}H_{2k}\big(\RP^n;\Z/2\Z\big)\to H_{2k}\big(\CP^n;\Z/2\Z\big)\end{equation}
induced by the canonical inclusion $\RP^n\to\CP^n$ does not vanish.
\end{lemma}
\begin{proof}[Proof of Lemma \ref{le:homol incl real proj}] We denote by $[\cdot]:\C^{n+1}\wo\{0\}\to\CP^n$ the canonical projection. We denote by $\mathbf{0}$ the origin in $\C^{n-2k}$ and define
\begin{align*}X:=\big\{&[x,\mathbf{0}]\,\big|\,x\in\R^{2k+1}\wo0\big\}\\
Y:=\big\{&\big[z_0,z_1,iz_1,\ldots,z_k,iz_k,z_{k+1},z_{k+2},\ldots,z_{n-k}\big]\,\big|\\
&(z_0,\ldots,z_{n-k})\in\C^{n-k+1}\wo0\big\}.
\end{align*}
These sets are closed real submanifolds of $\CP^n$ \footnote{They are diffeomorphic to $\RP^{2k}$ and $\CP^{n-k}$, respectively.} 
Denoting by $\mathbf{0}$ the origin in $\C^n$, we have
\[X\cap Y=\{[1,\mathbf{0}]\}.\]
This intersection is transverse, as follows from a calculation in standard charts. 
It follows that $X$ represents a nonzero $\Z/2\Z$-homology class. Since it is the image of the submanifold $\RP^{2k}$ of $\RP^n$ under the canonical inclusion $\RP^n\to\CP^n$, the statement of Lemma \ref{le:homol incl real proj} follows.
\end{proof}
\begin{proof}[Proof of Proposition \ref{prop:big asympt Lip} in the \textbf{case} \reff{prop:big asympt Lip:CP}]\setcounter{claim}{0} Let $(M,\om,U)$ be as in \reff{prop:big asympt Lip:CP}. We denote $L:=\RP^n\sub\CP^n$, respectively. Let $C\in(0,\infty)$. We choose a function $H\in C^\infty_c(U)$ such that
\[\int_UH\om^n=0,\quad -1\leq H\leq C,\quad H=C\textrm{ on }L.\]
It follows that 
\begin{equation}\label{eq:phi n H U}\Vert\phi_H^1\Vert\leq\VERT H\VERT=\max_MH-\min_MH\leq C+1.\end{equation}
(Recall that we use the abbreviated notation
\[\Vert\cdot\Vert:=\Vert\cdot\Vert^U_c:\Hamc(U)\to\R.)\]
\begin{claim}\label{claim:geq C} If $n=1$ then we have
\begin{equation}\label{eq:Vert phi M -pi}\Vert\wt{\phi_H^1}\Vert\geq C-\pi.\footnote{Here the tilde is defined as in \eqref{eq:phi wt phi}.}
\end{equation}
Otherwise we have
\begin{equation}\label{eq:Vert phi M}\Vert\wt{\phi_H^1}\Vert\geq C.
\end{equation}
\end{claim}
Since $C>0$ is arbitrary, the inequality \eqref{eq:Lip infty M U geq} and therefore the equality \eqref{eq:Lip infty M U 1} follow from this claim and \eqref{eq:phi n H U}.
\begin{proof}[Proof of Claim \ref{claim:geq C}] Consider the \textbf{case} in which $n=1$. The submanifold $L=\RP^1$ of $\CP^1$ is a stem in the sense of the definition on p.~775 in \cite{EPRigid}. Therefore, by \cite[Theorems 1.8 and 1.4]{EPRigid}, $L$ is stably non-displaceable. By \cite[7.2.A]{PoGeo} 
the fundamental group of $\Ham(\CP^1)$ \footnote{$\Ham(M)$ denotes the group of Hamiltonian diffeomorphisms of $M$, which agrees with $\Hamc(M)$ if $M$ is closed.} 
is isomorphic to $\Z_2$, and its nontrivial element $\ga$ is induced by the 1-turn rotation, where we view $\CP^1$ as the sphere $S^2$. This element has norm $\nu(\ga)$ (defined as in \cite[Definition 7.3.A]{PoGeo}) equal to $\pi$. 
Hence inequality \eqref{eq:Vert phi M -pi} follows from \cite[Theorem 7.4.A]{PoGeo}, using Definition 7.3.A in that book and the facts $\int_UH\om^n=0$ and $H=C$ on $L$.

Consider now the \textbf{case} in which $n\geq2$. To see that \eqref{eq:Vert phi M} holds, we denote by 
\[\mu:\Ham(\CP^n)\to\R\]
the Floer homological Calabi quasi-morphism of $\CP^n$ associated with the fundamental class $[\CP^n]$. See \cite[Sections 3.4 and 4.3]{EPCal}.\footnote{The construction of the map $\mu$ as in \cite[Sections 3.4 and 4.3]{EPCal} involves a unity of a factor in a splitting of $\QHev(M)$, the even-dimensional quantum homology. 
(See \cite[p.~1654]{EPCal}.) Since $\QHev(\CP^n)$ is a field, $[\CP^n]$ is indeed such a unity.} 
By \cite[Corollary 3.6]{EPCal} $\mu$ satisfies the bound
\begin{equation}\label{eq:mu phi}\Vert\phi\Vert\geq\frac{|\mu(\phi)|}{\int_{\CP^n}\om^n},\quad\forall\phi\in\Ham(\CP^n).
\end{equation}
We define
\begin{equation}\label{eq:ze}\ze:C^\infty(\CP^n)\to\R,\quad\ze(F):=\frac{\int_{\CP^n}F\om^n-\mu(\phi_F^1)}{\int_{\CP^n}\om^n}.\end{equation}
Real projective space $\RP^n$ is a closed monotone Lagrangian submanifold of $\CP^n$ with minimal Maslov number $N_{\RP^n}$ equal to $n+1$. This follows from \cite[Examples.~(i), p.~954]{Oh}. Since $n\geq2$, by Lemma \ref{le:homol incl real proj} the map \eqref{eq:H k} does not vanish for $k=1$. Since $2>\dim(\RP^n)+1-N_{\RP^n}=0$, this means that $\RP^n$ satisfies the Albers condition, as defined in \cite[p.~785]{EPRigid}. Therefore, by \cite[Theorem 1.17]{EPRigid} $\RP^n$ is $[\CP^n]$-heavy. (See \cite[Definition 1.3, p.~779]{EPRigid}.)

We define $\wt H:\CP^n\to\R$ by $\wt H(x):=H(x)$, if $x\in U$, and $\wt H(x):=0$, otherwise. Since $\RP^n$ is $[\CP^n]$-heavy and $\wt H=C$ on $\RP^n$, it follows that
\[\ze(\wt H)\geq C.\]
Combining this with the equality $\int\wt H\om^n=0$ and the definition \eqref{eq:ze} of $\ze$, we obtain 
\[-\frac{\mu(\phi_{\wt H}^1)}{\int_{\CP^n}\om^n}\geq C.\]
Combining this with the bound \eqref{eq:mu phi}, inequality \eqref{eq:Vert phi M} follows. This proves Claim \ref{claim:geq C}.
\end{proof}

This completes the proof of inequality \eqref{eq:Lip infty M U geq} and hence of equality \eqref{eq:Lip infty M U 1}. This proves Proposition \ref{prop:big asympt Lip} in the case \reff{prop:big asympt Lip:CP}.
\end{proof}
\begin{Rmks}\textnormal{
\begin{itemize}\item The above proof was suggested to us by L.~Polterovich (private communication).
\item The map $\ze$ defined in \eqref{eq:ze}, is a symplectic quasi-state. See \cite{EPQuasi}, definition (4) on p.~84 and the discussion afterwards.
\end{itemize}
}
\end{Rmks}
\subsection{Relative Hofer diameter}\label{subsec:rel Hofer diam}
In this subsection we explain the remark made after \eqref{eq:Diam c}. The diameter of a pseudo-distance function $d$ on a set $X$ is by definition the number 
\[\diam(d):=\sup\big\{d(x,y)\,\big|\,x,y\in X\big\}.\]
Let $(M,\om)$ be a symplectic manifold and $U\sub M$ an open subset. We can view $\Diam(U,M)$ (defined in \eqref{eq:Diam c}) as such a diameter, as follows. Let $G$ be a group. By a \emph{semi-norm}\label{semi-norm} on $G$ we mean a map $\Vert\cdot\Vert:G\to[0,\infty]$ such that
\begin{eqnarray}
\nn&\Vert\one\Vert=0,&\\
\nn&\Vert g^{-1}\Vert=\Vert g\Vert,&\\
\nn&\Vert gh\Vert\leq\Vert g\Vert+\Vert h\Vert,&
\end{eqnarray}
for every $g,h\in G$. We call the last of these conditions the \emph{triangle inequality}. We call $\Vert\cdot\Vert$ a \emph{norm} iff it is also nondegenerate, i.e., for every $g\in G$ it satisfies
\[\Vert g\Vert=0\then g=\one.\]
Every semi-norm $\Vert\cdot\Vert$ on $G$ gives rise to a pseudo-distance function $d(\Vert\cdot\Vert)$ on $G$ via
\[d(\Vert\cdot\Vert)(g,h):=\Vert g^{-1}h\Vert.\]
The diameter of $d(\Vert\cdot\Vert)$ is given by
\[\diam(d(\Vert\cdot\Vert))=\sup_{g\in G}\Vert g\Vert.\]
Consider now the canonical extension homomorphism $E:\Hamc(U)\to\Hamc(M)$ given by \eqref{eq:phi wt phi}. The map $\Vert\cdot\Vert^M_c\circ E:\Hamc(U)\to[0,\infty)$ is a norm. (Nondegeneracy follows from Theorem 1.1 in the article \cite{LMGeo} of D.~McDuff and F.~Lalonde.) The relative Hofer-diameter of $U$ in $M$ is given by the diameter of the distance function induced by this norm,
\[\Diam(U,M)=\diam\big(d\big(\Vert\cdot\Vert^M_c\circ E\big)\big).\]
\section{Proof of Proposition \ref{prop:H H'} (action for concatenated Hamiltonian)}\label{sec:proof:prop:H H'}
For the proof of Proposition \ref{prop:H H'}, we need the following. Let $(M,\om)$ be a symplectic manifold. For a function $H\in C^\infty_c([0,1]\x M)$ and $x_0\in M$ we denote
\begin{equation}\label{eq:D H x 0}\D^H_{x_0}:=\big\{u\in C^\infty(\D,M)\,\big|\,u(e^{2\pi it})=\phi_H^t(x_0),\,\forall t\in[0,1]\big\}.\end{equation}
\begin{lemma}[concatenated Hamiltonian and symplectic action]\label{le:Phi} Let $H,H'\in C^\infty_c\big([0,1]\x M\big)$ be such that $H_t,H'_t=0$ for $t$ in some neighbourhood of $\{0,1\}$, and defining $X:=\bigcup_{t\in[0,1]}\supp H_t$, condition \eqref{eq:phi H'} is satisfied. For every $x_0\in M$ there exists a bijection 
\begin{equation}\label{eq:Phi D}\Phi:\D^{H'}_{x_0}\to\D^{H\#H'}_{x_0},
\end{equation}
such that 
\begin{equation}\label{eq:int D Phi}\int_\D\Phi(u')^*\om=\int_\D{u'}^*\om,\quad\forall u'\in\D^{H'}_{x_0}.
\end{equation}
\end{lemma}
\begin{proof}[Proof of Lemma \ref{le:Phi}]\setcounter{claim}{0} We define the map $\Phi$ as follows. By hypothesis, there exists $\eps>0$ such that $H_t,H'_t=0$ for $t\in[0,2\eps]\cup[1-2\eps,1]$. We choose a diffeomorphism $\phi:\D\to\D$ such that 
\begin{equation}\label{eq:phi e}\phi(e^{2\pi it})=e^{2\pi i(2t-1)},\quad\forall t\in\left[\frac12+\eps,1-\eps\right].\end{equation}
Assume that $u'\in\D^{H'}_{x_0}$. We define 
\begin{equation}\label{eq:v u'}v:=\Phi(u'):=u'\circ\phi:\D\to M.\end{equation}
\begin{claim}\label{claim:v} We have $v\in\D^{H\#H'}_{x_0}$.
\end{claim}
\begin{proof}[Proof of Claim \ref{claim:v}] Since $H'_t=0$ for $t\in[0,2\eps]\cup[1-2\eps,1]$, we have 
\begin{equation}\label{eq:u' e}u'(e^{2\pi it})=\phi_{H'}^t(x_0)=x_0,\quad\forall t\in[-2\eps,2\eps].
\end{equation}
It follows from \eqref{eq:phi e} that 
\[\phi(e^{2\pi it})\in\big\{e^{2\pi it'}\,\big|\,t'\in[-2\eps,2\eps]\big\},\quad\forall t\in\left[0,\frac12+\eps\right]\cup\big[1-\eps,1\big].\]
Combining this with (\ref{eq:v u'},\ref{eq:u' e}), it follows that 
\begin{equation}\label{eq:v e 2 pi}v(e^{2\pi it})=x_0,\quad\forall t\in\left[0,\frac12+\eps\right]\cup[1-\eps,1].\end{equation}
Therefore, using again \eqref{eq:u' e}, we have
\begin{equation}\label{eq:v H H'}v(e^{2\pi it})=\phi_{H'}^{2t-1}(x_0)=\phi_{H\#H'}^t(x_0),\quad\forall t\in\left[\frac12,\frac12+\eps\right]\cup[1-\eps,1].
\end{equation}
Furthermore, it follows from \eqref{eq:phi H'} that 
\begin{equation}\label{eq:x 0 X}x_0\not\in X=\bigcup_{t\in[0,1]}\supp H_t. 
\end{equation}
This implies that 
\begin{equation}\label{eq:phi H t x 0 x 0}\phi_H^t(x_0)=x_0,\quad\forall t\in[0,1].\end{equation} 
Combining this with \eqref{eq:v e 2 pi}, it follows that 
\begin{equation}\label{eq:v e}v(e^{2\pi it})=\phi_H^{2t}(x_0)=\phi_{H\#H'}^t(x_0),\quad\forall t\in\left[0,\frac12\right].\end{equation}
Finally, it follows from (\ref{eq:phi e},\ref{eq:v u'}) and the fact $u'\in\D^{H'}_{x_0}$, that 
\[v(e^{2\pi it})=\phi_{H'}^{2t-1}(x_0)=\phi_{H\#H'}^t(x_0),\quad\forall t\in\left[\frac12+\eps,1-\eps\right].\]
Combining this with (\ref{eq:v H H'},\ref{eq:v e}), it follows that $v\in\D^{H\#H'}_{x_0}$. This proves Claim \ref{claim:v}.
\end{proof}
A similar argument shows that
\[v\circ\phi^{-1}\in\D^{H'}_{x_0},\quad\forall v\in\D^{H\#H'}_{x_0}.\]
It follows that the map $\Phi$ is a bijection. 

Equality \eqref{eq:int D Phi} follows from \eqref{eq:v u'}, using that $\phi$ is orientation preserving. This proves Lemma \ref{le:Phi}.
\end{proof}
\begin{proof}[Proof of Proposition \ref{prop:H H'}]\setcounter{claim}{0}\label{proof:prop:H H'} Statement \reff{prop:H H':PP} 
follows from Lemma \ref{le:Phi}. We prove statement \reff{prop:H H':A}. Assume that $(M,\om)$ is aspherical, and that $x_0\in\PP^\circ(H')$. It follows from the definition of $H\#H'$ that
\begin{equation}\label{eq:int H H'}\int_0^1(H\#H')_t\circ\phi_{H\#H'}^t(x_0)dt=\int_0^1H_t\circ\phi_H^t(x_0)dt+\int_0^1H'_t\circ\phi_{H'}^t(x_0)dt.
\end{equation}
Furthermore, it follows from \eqref{eq:phi H'} that 
\begin{equation}\label{eq:x 0 not in X} x_0\not\in X=\bigcup_{t\in[0,1]}\supp H_t. 
\end{equation}
This implies that $\phi_H^t(x_0)=x_0$, for every $t\in[0,1]$. Hence, using \eqref{eq:x 0 not in X} again, it follows that 
\begin{equation}\label{eq:H t phi H t x 0}H_t\circ\phi_H^t(x_0)=0,\quad\forall t\in[0,1].
\end{equation}
We choose a map $\Phi$ as in Lemma \ref{le:Phi} and a map $u'\in\D^{H'}_{x_0}$. (defined as in \eqref{eq:D H x 0}). The claimed equality \eqref{eq:A H H'} is a consequence of (\ref{eq:A H x0},\ref{eq:int H H'},\ref{eq:H t phi H t x 0},\ref{eq:int D Phi}). This proves \reff{prop:H H':A} and completes the proof of Proposition \ref{prop:H H'}.
\end{proof}

\bibliographystyle{amsalpha}
\bibliography{amsj,references}

\providecommand{\bysame}{\leavevmode\hbox to3em{\hrulefill}\thinspace}
\providecommand{\MR}{\relax\ifhmode\unskip\space\fi MR }
\providecommand{\MRhref}[2]{%
  \href{http://www.ams.org/mathscinet-getitem?mr=#1}{#2}
}
\providecommand{\href}[2]{#2}
\begin{thebibliography}{McD10b}

\bibitem[BIP08]{BIP}
Dmitri Burago, Sergei Ivanov, and Leonid Polterovich,
  \emph{Conjugation-invariant norms on groups of geometric origin}, Groups of
  diffeomorphisms, Adv. Stud. Pure Math., vol.~52, Math. Soc. Japan, Tokyo,
  2008, pp.~221--250. \MR{2509711}

\bibitem[BK17]{BK}
Michael Brandenbursky and Jarek K\k{e}edra, \emph{The autonomous norm on
  {$\operatorname{Ham}(\bold{R}^{2n})$} is bounded}, Ann. Math. Qu\'{e}.
  \textbf{41} (2017), no.~1, 63--65. \MR{3639648}

\bibitem[EP03]{EPCal}
Michael Entov and Leonid Polterovich, \emph{Calabi quasimorphism and quantum
  homology}, Int. Math. Res. Not. (2003), no.~30, 1635--1676. \MR{1979584}

\bibitem[EP06]{EPQuasi}
\bysame, \emph{Quasi-states and symplectic intersections}, Comment. Math. Helv.
  \textbf{81} (2006), no.~1, 75--99. \MR{2208798}

\bibitem[EP09]{EPRigid}
\bysame, \emph{Rigid subsets of symplectic manifolds}, Compos. Math.
  \textbf{145} (2009), no.~3, 773--826. \MR{2507748}

\bibitem[GS79]{GS}
R.~E. Greene and K.~Shiohama, \emph{Diffeomorphisms and volume-preserving
  embeddings of noncompact manifolds}, Trans. Amer. Math. Soc. \textbf{255}
  (1979), 403--414. \MR{542888}

\bibitem[HZ94]{HZ}
Helmut Hofer and Eduard Zehnder, \emph{Symplectic invariants and {H}amiltonian
  dynamics}, Birkh\"{a}user Advanced Texts: Basler Lehrb\"{u}cher.
  [Birkh\"{a}user Advanced Texts: Basel Textbooks], Birkh\"{a}user Verlag,
  Basel, 1994. \MR{1306732}

\bibitem[LM95a]{LMGeo}
Fran\c{c}ois Lalonde and Dusa McDuff, \emph{The geometry of symplectic energy},
  Ann. of Math. (2) \textbf{141} (1995), no.~2, 349--371. \MR{1324138}

\bibitem[LM95b]{LMHofer}
\bysame, \emph{Hofer's {$L^\infty$}-geometry: energy and stability of
  {H}amiltonian flows. {I}, {II}}, Invent. Math. \textbf{122} (1995), no.~1,
  1--33, 35--69. \MR{1354953}

\bibitem[L{\"u}c12]{Lu}
Wolfgang L{\"u}ck, \emph{Aspherical manifolds}, Bulletin of the Manifold Atlas
  (2012), 1--17.

\bibitem[McD02]{McD_geometric}
Dusa McDuff, \emph{Geometric variants of the {H}ofer norm}, J. Symplectic Geom.
  \textbf{1} (2002), no.~2, 197--252. \MR{1959582}

\bibitem[McD10a]{McD_Ham}
\bysame, \emph{Loops in the {H}amiltonian group: a survey}, Symplectic topology
  and measure preserving dynamical systems, Contemp. Math., vol. 512, Amer.
  Math. Soc., Providence, RI, 2010, pp.~127--148. \MR{2605315}

\bibitem[McD10b]{McDMono}
\bysame, \emph{Monodromy in {H}amiltonian {F}loer theory}, Comment. Math. Helv.
  \textbf{85} (2010), no.~1, 95--133. \MR{2563682}

\bibitem[MS12]{MSJ}
Dusa McDuff and Dietmar Salamon, \emph{{$J$}-holomorphic curves and symplectic
  topology}, second ed., American Mathematical Society Colloquium Publications,
  vol.~52, American Mathematical Society, Providence, RI, 2012. \MR{2954391}

\bibitem[MS17]{MSIntro}
\bysame, \emph{Introduction to symplectic topology}, third ed., Oxford Graduate
  Texts in Mathematics, Oxford University Press, Oxford, 2017. \MR{3674984}

\bibitem[Oh93]{Oh}
Yong-Geun Oh, \emph{Floer cohomology of {L}agrangian intersections and
  pseudo-holomorphic disks. {I}}, Comm. Pure Appl. Math. \textbf{46} (1993),
  no.~7, 949--993. \MR{1223659}

\bibitem[Ost03]{Os}
Yaron Ostrover, \emph{A comparison of {H}ofer's metrics on {H}amiltonian
  diffeomorphisms and {L}agrangian submanifolds}, Commun. Contemp. Math.
  \textbf{5} (2003), no.~5, 803--811. \MR{2017719}

\bibitem[Pol01]{PoGeo}
Leonid Polterovich, \emph{The geometry of the group of symplectic
  diffeomorphisms}, Lectures in Mathematics ETH Z\"{u}rich, Birkh\"{a}user
  Verlag, Basel, 2001. \MR{1826128}

\bibitem[PS21]{PS}
Leonid Polterovich and Egor Shelukhin, \emph{Lagrangian configurations and
  hamiltonian maps}, arXiv:2102.06118 (2021).

\bibitem[Sch00]{Schwa}
Matthias Schwarz, \emph{On the action spectrum for closed symplectically
  aspherical manifolds}, Pacific J. Math. \textbf{193} (2000), no.~2, 419--461.
  \MR{1755825}

\bibitem[Sik90]{Si}
Jean-Claude Sikorav, \emph{Syst\`{e}mes hamiltoniens et topologie
  symplectique}, Dipartimento di Mathematica dell' Universit\`{a} di Pisa, ETS,
  EDITRICE PISA, 1990.

\bibitem[Sti93]{St}
John Stillwell, \emph{Classical topology and combinatorial group theory},
  second ed., Graduate Texts in Mathematics, vol.~72, Springer-Verlag, New
  York, 1993. \MR{1211642}

\bibitem[SZ12]{SZSmall}
Jan Swoboda and Fabian Ziltener, \emph{Coisotropic displacement and small
  subsets of a symplectic manifold}, Math. Z. \textbf{271} (2012), no.~1-2,
  415--445. \MR{2917151}

\bibitem[Zil10]{ZiLeafwise}
Fabian Ziltener, \emph{Coisotropic submanifolds, leaf-wise fixed points, and
  presymplectic embeddings}, J. Symplectic Geom. \textbf{8} (2010), no.~1,
  95--118. \MR{2609631}

\end{thebibliography}

\end{document}